\documentclass{amsart}

\usepackage{amssymb}
\usepackage{amsmath,amscd}
\usepackage{graphicx}
\usepackage{color,hyperref}
\usepackage{tikz}
\usetikzlibrary{matrix,arrows}

\newcommand{\mf}{\mathfrak}
\newcommand{\mc}{\mathcal}

\newcommand{\R}{\mathbf R}
\newcommand{\C}{\mathbf C}
\newcommand{\Q}{\mathbf Q}
\newcommand{\Z}{\mathbf Z}
\newcommand{\F}{\mathbf F}

\newcommand{\oo}{\mathcal{O}}

\newcommand{\Fix}{\textnormal{Fix}}

\numberwithin{equation}{section}

\theoremstyle{plain}
\newtheorem{theorem}{Theorem}[section]
\newtheorem{lemma}[theorem]{Lemma}

\newtheorem{corollary}[theorem]{Corollary}
\newtheorem{theorem*}{Theorem}
\newtheorem{lemma*}{Lemma}
\newtheorem{corollary*}{Corollary}
\newtheorem{proposition*}{Proposition}

\theoremstyle{definition}

\newtheorem{example}[theorem]{Example}

\begin{document}

\title{Value sets of bivariate Chebyshev maps over finite fields}

\author{\"{O}mer K\"{u}\c{c}\"{u}ksakall{\i}}
\address{Middle East Technical University, Mathematics Department, 06800 Ankara,
Turkey.}
\email{komer@metu.edu.tr}

\date{\today}

\begin{abstract}
We determine the cardinality of the value sets of bivariate Chebyshev maps over 
finite fields. We achieve this using the dynamical properties of these maps 
and the algebraic expressions of their fixed points in terms of roots of unity.
\end{abstract}



\maketitle

\section*{Introduction}
The Chebyshev polynomials show remarkable properties and they have applications 
in many areas of mathematics. There is a generalization of these polynomials to 
several variables introduced by Lidl and Wells \cite{lidlwells}. It is a well 
known fact that the Dickson polynomial, a normalization of one variable 
Chebyshev polynomial, induces a permutation on $\F_q$ if and only if 
$\gcd(k,q^s-1)=1$ for $s=1,2$. It is in perfect analogy with one variable case 
that the $n$ variable Chebyshev map is a bijection of $\F_q^n$ if and only if 
$\gcd(k,q^s-1)=1$ for $s=1,2,\ldots,n+1$. 

Let $f:\F_q^n\rightarrow\F_q^n$ be a polynomial map in $n$ variables defined 
over $\F_q$. Denote its value set by $V(f,\F_q^n)=\{f(c):c\in\F_q^n\}$. Clearly 
$f$ is a bijection of $\F_q^n$ if and only if its value set has cardinality 
$q^n$. If $f$ is not a bijection, then it is natural to ask how far it is away 
from being a bijection. There are several results in the literature which give 
bounds on the cardinality of the value set. We refer to the work of Mullen, Wan 
and Wang \cite{mullen} for a nice introduction to this problem which include 
several references and historical remarks.

For an arbitrary polynomial map, there is no easy formula giving the cardinality 
of the value set. However Chou, Gomez-Calderon and Mullen \cite{chou} achieve 
in finding such a formula for the Dickson polynomials. In our previous work 
\cite{kucuksakalli}, we gave a shorter proof of this formula by using a singular 
cubic curve and generalized those computations to the elliptic case. More 
precisely we have found the cardinality of the value sets of Latt\`{e}s maps, 
which are induced by isogenies of elliptic curves, over finite fields. 

In this paper we study the Chebyshev maps with two variables. The bivariate 
Chebyshev map $\mc{T}_k$ is given by the formula
\[ \mc{T}_k(x,y) = (g_k(x,y),g_k(y,x)) \]
where $g_k(x,y)$ is the generalized Chebyshev polynomial defined by Lidl 
and Wells \cite{lidlwells}. We have $g_{-1}(x,y)=y, g_{0}(x,y)=3 $ and 
$g_{1}(x,y)=x$ and these polynomials satisfy the recurrence relation
\[ g_k(x,y) = xg_{k-1}(x,y) - yg_{k-2}(x,y) + g_{k-3}(x,y). \]
The recurrence relation work in both ways and $g_{k}(x,y)$ is defined for all 
integers $k\in\Z$. Note that $g_k$ has integral coefficients and one can 
consider the map induced on finite fields. The main result of this paper is 
Theorem~\ref{main} which provides a formula for the cardinality of the 
value set 
\[V(\mc{T}_k,\F_q^2)=\{\mc{T}_k(x,y):(x,y)\in\F_q^2\}.\]
We achieve in finding such a formula by using the dynamical properties of 
bivariate Chebyshev maps over complex numbers which are studied by Uchimura 
\cite{uchimura}. Uchimura shows that the set of points with bounded orbits is a 
certain closed domain $S$ in $\C^2$. This set is enclosed by Steiner's 
hypocycloid, see Figure~\ref{fig:domains}. Moreover he shows that the 
number of periodic points of order $n$ is equal to $|k|^{2n}$ if $|k|\geq2$. 
Another tool for our computations is the nice expression of periodic points of 
$\mc{T}_k$. Algebraically each periodic point is a triple sum of roots of unity, 
a characterization due Koornwinder \cite{koornwinder}. We combine these facts 
together with the identity $\mc{T}_q(x,y) \equiv (x^q,y^q) \pmod{p}$.It follows 
that $q^2$ fixed points of $\mc{T}_q$ reduce 
to distinct elements in $\F_q^2$ modulo a certain prime ideal of a number field. 
This is the idea we have used in order to compute the size of the value sets for 
Latt\`{e}s maps \cite{kucuksakalli}. After characterizing the elements in 
$\F_q^2$ in a compatible fashion under the action of $\mc{T}_q$, determining the 
cardinality of $V(\mc{T}_k,\F_q^2)$ reduces to a combinatoric argument.

The organization of the paper is as follows: In the first section we 
give an alternative computation of the value set of Dickson polynomials in 
order to give the idea of our computations in the bivariate case. In the second 
section, we review some known facts about the dynamics of bivariate Chebyshev 
maps and classify the points which have bounded orbits. In the 
third section we focus on the periodic and preperiodic points of $\mc{T}_k$ 
over complex numbers and their algebraic expressions. In the last section, we 
find the cardinality of the value set of bivariate Chebyshev maps. We finish 
our paper by giving an example.

\section{Single variable case}
In this section, we will consider the Dickson polynomials, a normalization of 
one variable Chebyshev polynomials, and give an alternative computation of the 
cardinality of their value sets. This alternative computation will be a summary 
of the ideas that will be used in the rest of the paper. 

The family of Chebyshev polynomials (of the first kind) are defined by the 
recurrence relation $T_{k+1}(x) = 2xT_k(x) - T_{k-1}(x)$ where $T_0(x)=1$ and 
$T_1(x)=x$. One can normalize the Chebyshev polynomials by the relation 
$D_k(x)=2T(\frac{x}{2})$ to obtain the Dickson polynomials (of the first 
kind). This is done in order to cancel the repeating factors of two and 
consider the arithmetic over fields of characteristic two. The Dickson 
polynomials satisfy a similar recurrence relation $D_{k+1}(x) = xD_{k}(x) - 
D_{k-1}(x)$ where $D_0(x)=2$ and $D_1(x)=x$. 

The first few Chebyshev and Dickson polynomials are:
\begin{align*}
 T_0(x)&=1, & D_0(x) &=2,\\
 T_1(x)&=x, & D_1(x) &=x,\\
 T_2(x)&=2x^2-1, & D_2(x) &=x^2-2,\\
 T_3(x)&=4x^3-3x, & D_3(x) &=x^3-3x,\\
 T_4(x)&=8x^4-8x^2+1, & D_4(x) &=x^4-4x^2+2,\\
 T_5(x)&=16x^5-20x^3+5x, &D_5(x) &=x^5 - 5x^3 + 5x.
\end{align*}

The dynamics of Dickson polynomials $D_k(x)$ over complex numbers is 
well-understood. We refer to Silverman \cite{sil-dyn} for a nice summary of 
these results. In contrast with our terminology, Silverman uses the term 
Chebyshev polynomials for the family $D_k$. 

In order to emphasize the analogy between the one variable case and the two 
variable case, let us set 
\[\alpha(\sigma)=e^{2\pi i \sigma}+e^{-2\pi i \sigma}, \hspace{20pt} 
\sigma\in\R.\]
Let $k$ be a fixed integer with $k\geq 2$. It is a well-known fact that the 
Julia set of $D_k(x)$ is the closed interval $J=[-2,2]$ in $\C$. Observe that 
the Julia set can be given as $J=\{ \alpha(\sigma) : \sigma\in \R \}$. The 
(forward) orbit of $x$ under $D_k$ is the set $\oo(x)= \{D_k^n(x):n\geq0\}$ by 
definition. Observe that the interval $[-2,2]$ can also be obtained as the set 
of complex numbers $x$ whose orbits $\oo(x)$ are bounded sets. Note that any 
preperiodic or periodic point must be in the set $[-2,2]$ since their orbits 
have finitely many elements. 

The value sets of Dickson polynomials are first computed by Chou, 
Gomez-Calderon and Mullen \cite{chou}. In \cite{kucuksakalli}, we gave an 
alternative computation of their result by using a singular cubic curve. Now we 
will give another approach which is a summary of ideas that will be used in the 
rest of the paper.

The map $\alpha:\R/\Z \rightarrow [-2,2]$ is a two to one covering with two 
exceptional points. Namely the points $-2=\alpha(1/2)$ and $2=\alpha(0)$. The 
fixed points of $D_k$ satisfy the relation $D_k(x)=x$ by definition. Moreover 
$x=\alpha(\sigma)$ for some $\sigma\in\R$ and we have $D_k(\alpha(\sigma)) = 
\alpha(k\sigma)$. For real numbers $\sigma$ and $\tilde{\sigma}$, we have 
$\alpha(\sigma) = \alpha(\tilde{\sigma})$ if and only if $\{\sigma,1-\sigma\}$ 
and $\{\tilde{\sigma},1-\tilde{\sigma}\}$ are equal as subsets of $\R/\Z$. It 
follows from these observations that any fixed point of $D_k$ is of the form 
$x=\alpha(r)$ for some rational number $r$. Moreover if $r$ is written in its 
lowest terms, its denominator must be relatively prime to $k$. Using this 
characterization, we can write
\[\Fix(D_k,\C)=\left\{\alpha\left(\frac{a}{k-1}\right):a\in\Z\right\} \cup 
\left\{\alpha\left(\frac{a}{k+1}\right):a\in\Z\right\}.\]

Now let us count the elements in $\Fix(D_k,\C)$. Both sets in the above union
contain the element $2=\alpha(0)$. If $k$ is odd, then $-2$ is in their 
intersection as well. Other than these two elements, the above sets are disjoint 
since $\gcd(k-1,k+1)|2$. It follows easily that there are $k$ distinct elements 
in $\Fix(D_k,\C)$ by the inclusion and exclusion principle. 

Let $\F_q$ be a finite field of characteristic $p$. Consider the number field $K 
= \Q( \Fix(D_q,\C) )$ which is obtained by adjoining the fixed points of $D_q$ 
to the rational numbers. Let $\mf{p}$ be a prime ideal of $K$ lying over $p$. 
We have $D_q(x)\equiv x^q \pmod{p}$ and therefore the fixed points of $D_q$ 
reduced modulo $\mf{p}$ are the solutions of $x^q-x=0$. Thus each element of 
$\F_q$ is obtained by reducing a fixed point modulo $\mf{p}$. Since there are 
$q$ fixed points of $D_q$, we conclude that there is a one-to-one correspondence
\[\Fix(D_q,\C) \longleftrightarrow \F_q \]
which is obtained by the reduction modulo $\mf{p}$.

From this point on finding a formula for the size of the value set is 
straightforward. One can use the representations $\alpha(a/(q\pm1))$ of 
elements 
in $\Fix(D_q,\C)$ in order to analyze the value set of $D_k$. Applying the 
inclusion and exclusion principle, we find that
\[|V(D_k,\F_q)|=\frac{q-1}{2\gcd(k,q-1)}+\frac{q+1}{2\gcd(k,q+1)}+\eta(k,q).\]
Here $\eta(k,q)$ is a function which takes the values $0$ or $1/2$. More 
precisely, $\eta(k,q)=0$ if and only if $\gcd(k,q-1) \equiv \gcd(k,q+1) 
\pmod{2}$.

\section{Bivariate Chebyshev maps}
There is a generalization of Chebyshev maps to higher dimensions introduced by 
Lidl and Wells \cite{lidlwells}. For any integer $n$, the polynomial 
equation $z^2-nz+1=0$ has roots $y$ and $1/y$ in the complex numbers. If $y^k$ 
and $1/y^k$ are the roots of $z^2-n'z+1=0$ then $n'$ is also an integer, and 
it is a well known fact that $D_k(n)=n'$ where $D_k$ is the Dickson 
polynomial (of the first kind).

Lidl and Wells generalize this construction by considering a polynomial 
equation of degree $n+1$ with integral coefficients and with roots $t_1,t_2, 
\ldots, t_{n+1}$. Then they consider another polynomial equation with 
roots $t_1^k,t_2^k \ldots, t_{n+1}^k$. It turns out that there 
is a system of polynomials with integral coefficients which give the 
coefficients of the latter equation in terms of the coefficients of the former 
equation. They are called the generalized Chebyshev polynomials. For details 
see \cite{lidlwells}, or \cite{lidlnied}.

Now we focus on the bivariate case. Suppose that $x=t_1+t_2+t_3$ and $y=t_1t_2 + 
t_1t_3 + t_2t_3$ with $t_1t_2t_3=a$. We assume that $a=1$ for simplicity. 
According to the construction of Lidl and Wells, there exists a bivariate 
polynomial $g_k(x,y)$ with integer coefficients which maps $(x,y)$ to $t_1^k 
+ t_2^k + t_3^k$. Moreover it turns out that $g_k(y,x)$ is equal to $t_1^kt_2^k 
+ t_1^kt_3^k + t_2^kt_3^k$. It is easy to see that $g_{-1}(x,y)=y, g_0(x,y)=3$ 
and $g_1(x,y)=x$. Further this family satisfies the recurrence relation 
\[ g_k(x,y) = xg_{k-1}(x,y) - yg_{k-2}(x,y) + g_{k-3}(x,y). \]

The first few bivariate Chebyshev polynomials are:
\begin{align*}
 g_0(x,y)&=3,\\
 g_1(x,y)&=x,\\
 g_2(x,y)&=x^2-2y,\\
 g_3(x,y)&=x^3-3xy+3,\\
 g_4(x,y)&=x^4- 4x^2y + 2y^2+4x,\\
 g_5(x,y)&=x^5 - 5x^3y + 5xy^2 + 5x^2 - 5y.
 \end{align*}

As we have seen from the first section, the dynamical properties of Dickson 
polynomials play an important role in the analysis of the map induced over 
finite fields. Thus we start with reviewing some known facts about the 
bivariate Chebyshev map which is defined by
\[ \mc{T}_k(x,y) = (g_k(x,y),g_k(y,x)). \]

Dynamical properties of $\mc{T}_k$ on complex numbers are studied by Uchimura 
\cite{uchimura}. Uchimura shows that the map $\mc{T}_k$ admits an invariant 
plane $\{x=\bar{y}\}\subseteq\mathbf{C}^2$. The restriction of $\mc{T}_k$ to 
this plane is the polynomial considered by Koornwinder \cite{koornwinder}. With 
Koornwinder's notation, we have $\mc{T}_k(x,\bar{x}) = P_{(k,0)}^{-1/2} 
(x,\bar{x})$. The key property we get from Koornwinder's work is the nice action 
of Chebyshev maps on certain elements. Define
 \[ \alpha(\sigma,\tau)=e^{2\pi i\sigma}+e^{2\pi i\tau}+e^{2\pi i 
(-\sigma-\tau)}, \hspace{20pt} \sigma,\tau\in\R.\]
We have
\[ \mc{T}_k\left(\alpha(\sigma,\tau), \overline{\alpha(\sigma,\tau)} \right)= 
(\alpha(k\sigma,k\tau),\overline{\alpha(k\sigma,k\tau)}). \]

Let $k$ be a fixed integer with $|k|\geq2$. Uchimura shows that any point in 
$\C^2$, whose orbit under $\mc{T}_k$ is a bounded set, must be in $\{(x, 
\bar{x}): x\in S\}$ where
\[ S=\{\alpha(\sigma,\tau):\sigma,\tau\in\R\}. \]
Note that any periodic or preperiodic point must have a bounded orbit and 
therefore it must be in $\{(x, \bar{x}): x\in S\}$ too. If we write $x=u+vi$, 
then the set $S$ is a closed domain enclosed by Steiner's hypocycloid
\[ (u^2+v^2+9)^2+8(-u^3+3uv^2)-108=0. \]
The Steiner's hypocycloid is a simple closed curve which can also be  
parametrized by $\alpha(\sigma,\sigma)$ with $0 \leq \sigma \leq 1$. 
\begin{figure}[htbp]
    \centering
 \includegraphics[scale=0.5]{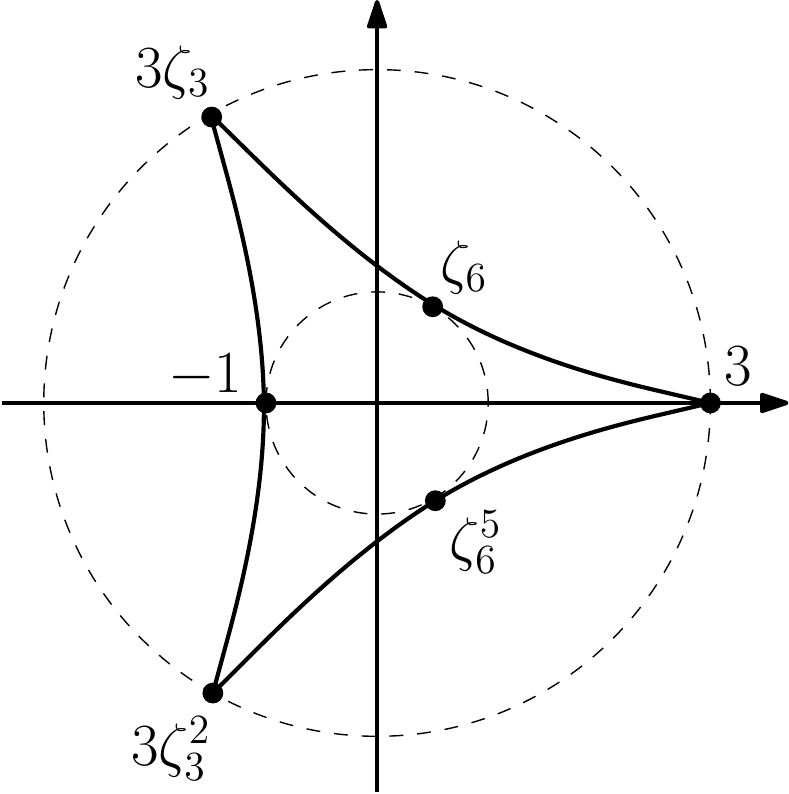}
    \caption{The domain $S$.}
    \label{fig:domains}
\end{figure}
There is a  symmetry under multiplication by a third root of unity because we 
have $\alpha(\sigma,\tau)\zeta_3 = \alpha(\sigma+1/3,\tau+1/3)$. 

Periodic points will play an important role in our computations. A periodic 
point must have a bounded orbit thus its coordinates are of the form 
$\alpha(\sigma,\tau)$ for some $\sigma,\tau\in\R$. The following lemma is the 
key to count the elements in the value sets of bivariate Chebyshev maps over 
finite fields.

\begin{lemma}\label{equal}
The complex numbers $\alpha(\sigma,\tau)$ and $\alpha(\tilde{\sigma} , 
\tilde{\tau})$ are equal if and only if $\{\sigma, \tau, -(\sigma+\tau)\}$ and 
$\{\tilde{\sigma}, \tilde{\tau}, - (\tilde{\sigma} + \tilde{\tau})\}$ are equal 
as subsets of $\R/\Z$.
\end{lemma}
\begin{proof}
To understand the representations of elements in $S$ in terms of 
$\alpha(\sigma,\tau)$, a useful idea is to consider the tangent lines to 
the hypocycloid 
\[C=\{\alpha(\sigma,\sigma):\sigma\in [0,1]\}.\]
Define $\ell_\sigma$ to be the line passing through the point 
$\alpha(\sigma,\sigma)$ with slope $-\tan(\pi\sigma)$. If $\sigma\in 1/2+\Z$, 
then set $\ell_\sigma$ as the vertical line $\mf{Re}(z)=-1$. Note that the 
lines 
$\ell_\sigma$ are distinct for $\sigma\in [0,1)$.

The subset $\{\sigma, \tau, -(\sigma+\tau)\}$ of $\R/\Z$ has one element if 
$\alpha(\sigma,\tau)$ is one of the three corner points of $C$. It has two 
elements if $\alpha(\sigma,\tau)$ is on $C$ but not a corner point. The lemma 
is trivially true in these cases.

We assume that $\alpha(\sigma,\tau)$ is an interior point of $S$. In this case 
the subset $\{\sigma, \tau, -(\sigma+\tau)\}$ of $\R/\Z$ has three elements 
and the lines $\ell_\sigma$, $\ell_\tau$ and $\ell_{-(\sigma+\tau)}$ are 
pairwise distinct. Observe that each interior point of $S$ is realized 
precisely 
three times by the lines $\ell_\sigma$ as $\sigma$ varies on the interval 
$[0,1)$. We will show that the lines $\ell_\sigma$, $\ell_\tau$ and 
$\ell_{-(\sigma+\tau)}$ intersect at $\alpha(\sigma,\tau)$. This geometric 
result will finish the proof because it gives a one-to-one correspondence 
between the interior points of $S$ and the subsets of $\R/\Z$ with three 
elements.

Consider $L_\sigma=\{ (\sigma,t):t\in\R \}$, a vertical line in $\R^2$. We 
claim that $\alpha$, regarded as a map from $\R^2$ to $\C$, maps $L_\sigma$ to 
the line segment $\ell_\sigma \cap S$. It follows that $\alpha(\sigma,\tau)$ 
lies on $\ell_\sigma$. By symmetry $\alpha(\sigma,\tau)$ lies on $\ell_\tau$, 
too. Therefore the lines $\ell_\sigma$ and $\ell_\tau$ intersects at 
$\alpha(\sigma,\tau)$. It is clear that $\ell_{-(\sigma+\tau)}$ passes through 
the same point because $\alpha(\sigma,\tau)=\alpha(\sigma,-(\sigma+\tau))$. To 
justify the claim $\alpha(L_\sigma) = \ell_\sigma \cap S$, we start with
\[ \alpha(L_\sigma) = \{ \zeta^{2 \pi i\sigma}+\zeta^{2\pi it}+\zeta^{-2\pi 
i(\sigma+t)}:t\in \R \}.\]
The parametric curve $\alpha(L_\sigma)$ in $\C$ has the following components:
\begin{align*}
 f(t)=\mf{Re}(\alpha(L_\sigma)) &= \cos(2\pi \sigma)+\cos(2\pi t) + 
\cos(2\pi(\sigma+t))\\
g(t)=\mf{Im}(\alpha(L_\sigma)) &= \sin(2\pi \sigma)+\sin(2\pi t) - 
\sin(2\pi(\sigma+t)).
\end{align*}
The slope of the tangent line to the curve  $\alpha(L_\sigma)$ at any point 
$\alpha(\sigma,t)$ is given by $m=g'(t)/f'(t)$ provided that $f'(t)\neq 
0$. We have
\begin{align*}
 f'(t)/(2\pi)&=-\sin(2\pi t)-\sin(2\pi(\sigma+t)=-2\sin(\pi(\sigma 
+ 2t))\cos(\pi \sigma)\\
 g'(t)/(2\pi)&=\cos(2\pi t)-\cos(2\pi(\sigma+t))=2\sin(\pi(\sigma + 
2t))\sin(\pi \sigma).
\end{align*}
Here, the second equalities follow from the sum to product formulas for the 
trigonometric functions. Thus, $m=-\tan(\pi \sigma)$. This computation shows 
that $L_\sigma$ is mapped under $\alpha$ to a line segment with slope 
$-\tan(\pi 
\sigma)$. If $\sigma\in 1/2+\Z$, then it is mapped to the vertical line 
$\mf{Re}(z)=-1$. Moreover, $\alpha(L_\sigma)$ is a line segment with end 
points lying on the hypocycloid $C$. To see this, note that the functions 
$f(t)$ 
and $g(t)$ have common critical values if and only if $\sin( \pi(\sigma+2t) 
)=0$. This is possible only at $t=1/2-\sigma/2$ and $t=1-\sigma/2$ for those 
$t\in[0,1)$. The points corresponding to these two $t$ values are on the 
hypocycloid $C$.
\end{proof}

As a final observation note that the restricted map $\alpha:[0,1)\times[0,1) 
\rightarrow S$ is a six to one map unless $\alpha(\sigma,\tau)$ is on the 
boundary $C=\{\alpha(\sigma,\sigma):\sigma\in \R\}$. A point on $C$ which is 
not a corner point can be represented in three different ways and the corner 
points $\alpha(0,0),\alpha(1/3,1/3)$ and $\alpha(2/3,2/3)$ can be represented 
uniquely. We finish this section with an illustration of the lines $\ell_\sigma$ 
within the proof of Lemma\ref{equal}.
\begin{example}
The point $-1+\sqrt{-3}\in S$ can be represented by any of the 
following six expressions: $\alpha(1/6,1/3), \alpha(1/3,1/6), \alpha(1/6,1/2), 
\alpha(1/2,1/6), \alpha(1/2,1/3)$ and $\alpha(1/3,1/2)$. Note that the lines 
$\ell_{1/6}, \ell_{1/3}$ and $\ell_{1/2}$ intersect at the same point, namely 
$-1+\sqrt{-3}$. See Figure~\ref{fig:tangency}.
\end{example}
\begin{figure}[htbp]
    \centering
\includegraphics[scale=0.4]{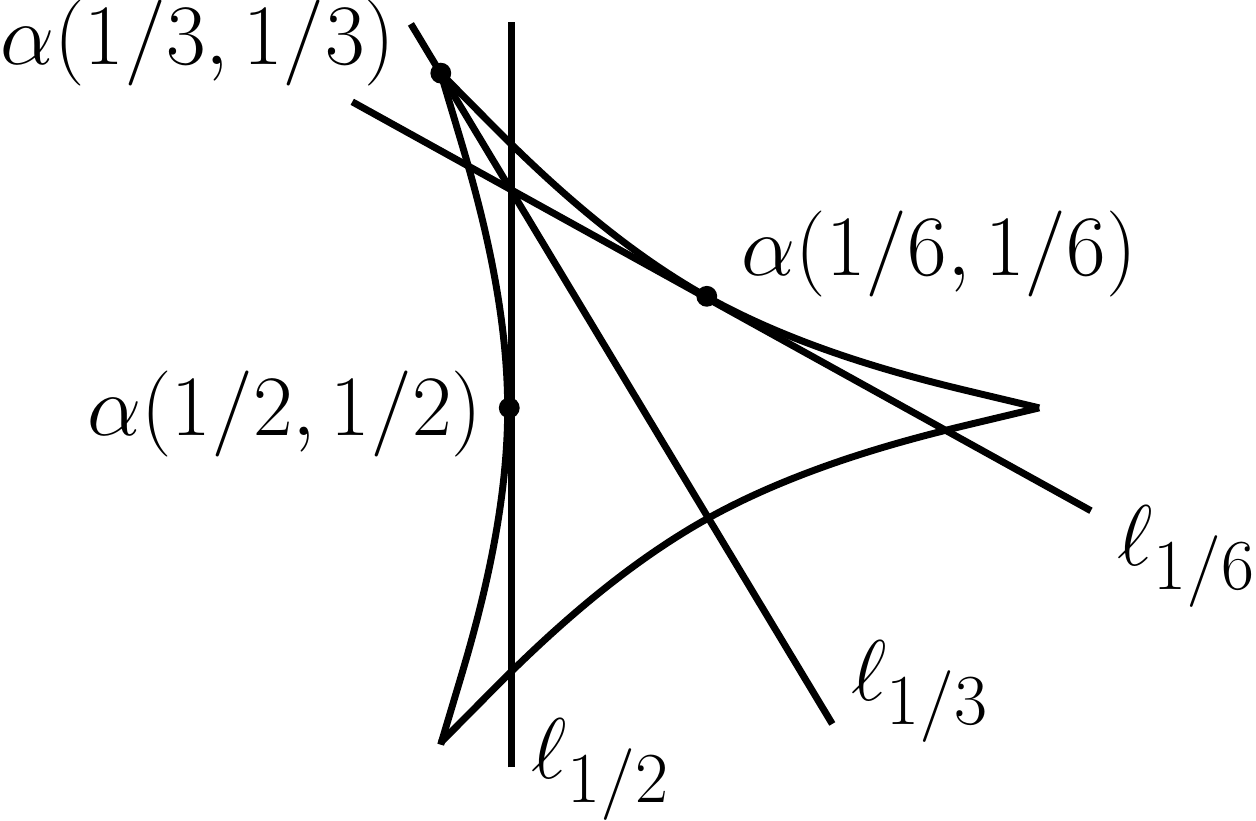}
    \caption{There lines meeting at $-1+\sqrt{-3}$.}
    \label{fig:tangency}
\end{figure}

\section{Periodic and preperiodic points}

The family of Dickson polynomials $D_k(x)$ has very explicit dynamical 
properties. For example a point with bounded orbit must lie in the interval 
$[-2,2]$ in $\C$. Moreover a point $x$ is preperiodic if and only if 
$x=2\cos(2\pi r)$ for some rational number $r$.

We want to classify all periodic and preperiodic points of $\mc{T}_k$. The 
cases $k=-1$ $k=0$ and $k=1$ are trivial so we assume that $|k| \geq 2$.  All 
points with bounded orbit lie in $\{(x,\bar{x}):x\in S\}$ and their coordinates 
are of the form $\alpha(\sigma,\tau)$ for some real numbers $\sigma$ and 
$\tau$. Periodic and preperiodic points have bounded orbits since there are 
finitely many elements in their orbits. Thus their coordinates are given by 
$\alpha(\sigma,\tau)$. Moreover we have the following

\begin{lemma}
Let $k$ be a fixed integer with $|k|\geq 2$. A point $(x,y)\in\C^2$ is a 
preperiodic point of $\mc{T}_k$ if and only if there exist rational 
numbers $r,s\in\Q$ such that $x=\alpha(r,s)$ and $y=\bar{x}=\alpha(-r,-s)$. 
Moreover if $r$ and $s$ are written in their lowest terms then 
$(\alpha(r,s),\alpha(-r,-s))$ is a periodic point of $\mc{T}_k$ if and only if 
the denominators of $r$ and $s$ are both relatively prime to $k$.
\end{lemma}
\begin{proof}
Suppose that $x=\alpha(r,s)$ where $r$ and $s$ are rational numbers. Then it is 
easy to see that $(x,\bar{x})$ is a preperiodic point of $\mc{T}_k$. For the 
converse, let $\alpha(\sigma,\tau)$ be a preperiodic point under $\mc{T}_k$ with 
$|k| \geq 2$. It follows that 
$$\alpha(k^n\sigma,k^n\tau) = \alpha(k^l\sigma,k^l\tau)$$ 
for some positive integers $n\leq l$. This is possible if and only if 
\[\{k^n\sigma, k^n\tau, -k^n(\sigma+\tau)\} = \{k^l\sigma, k^l\tau, 
-k^l(\sigma+\tau) \}\]
as subsets of $\R/\Z$ by Lemma~\ref{equal}. There are six possibilities. We 
will prove only one case. The proofs for the other cases are similar. Suppose 
that we have 
\begin{align*}
 k^n\sigma & \equiv k^l\tau \pmod{\Z}, \\
 k^n\tau & \equiv -k^l(\sigma+\tau) \pmod{\Z}.
\end{align*}
We omit the third equation since it can be obtained from these two. Starting 
with the former equation and then using the latter equation, we obtain
\[k^n\sigma \equiv k^l\tau \equiv k^{l-n}k^n\tau \equiv 
k^{l-n}[-k^l(\sigma+\tau)] \pmod{\Z}.\] 
Now we plug in the first equation again and get
\[k^n\sigma \equiv k^{l-n}[-k^l\sigma-k^n\sigma] \pmod{\Z}. \]
Therefore
\[ k^n\sigma+k^{2l-n}\sigma+k^l\sigma\equiv 0 \pmod{\Z}. \]
From this congruence, we see that $\sigma$ is a rational number with 
denominator $k^{2l-n}+k^l+k^n$. Since $k^n\sigma \equiv k^l\tau \pmod{\Z}$, the 
number $\tau$ must be rational too.

Now suppose that $(x,y)$ is a periodic point under $\mc{T}_k$. Then there exist 
rational numbers $r=a/b$ and $s=c/d$ for some integers $a,b,c$ and $d$ such that 
$x=\alpha(r,s)$ and $y=\alpha(-r,-s)$. Suppose that $r$ and $s$ are written in 
their lowest terms, i.e. $\gcd(a,b)=1$ and $\gcd(c,d)=1$. Without loss of 
generality we can assume that $(x,y)$ is fixed by $\mc{T}_k$. The general result 
will follow from the identity $\mc{T}_k\circ \mc{T}_m = \mc{T}_{km}$ which is 
valid on $\{ (x,\bar{x}):x\in S \}$. It follows by Lemma~\ref{equal} that the 
set $\{r,s,-(r+s)\}$ modulo $\Z$ is permuted under multiplication by $k$. Thus 
$r\equiv k^6r \pmod{\Z}$ and therefore $r(k^6-1)\equiv 0 \pmod{\Z}$. From this 
we conclude that the denominator of $r$ is relatively prime to $k$ since it is 
a 
divisor of $k^6-1$. The same result holds for $s$ as well. To see the converse 
let $f$ be order of $k$ modulo the least common multiple of denominators of $r$ 
and $s$. Then $\mc{T}_k^f$ fixes the point $(x,y)$ and therefore it is an 
$f$-periodic point of $\mc{T}_k$.
\end{proof}

Now we want to describe the set of points in $\C$ which are fixed under 
$\mc{T}_k$. Consider the following sets for $|k|\geq 2$:
\begin{align*}
 \mc{A}_k & =\left\{\alpha\left( \frac{d}{k-1}, \frac{e}{k-1} 
\right):d,e\in\Z \right\},\\ 
\mc{B}_k & =\left\{\alpha\left( \frac{d}{k^2-1}, \frac{dk}{k^2-1} 
\right):d\in\Z \right\},\\
\mc{C}_k & =\left\{\alpha\left( \frac{d}{k^2+k+1} , \frac{dk}{k^2+k+1} 
\right):d\in\Z \right\}.
\end{align*}
It is obvious that $\Fix(\mc{T}_k,\C^2) \supseteq \{ (x, \bar{x}): 
x\in \mc{A}(k) \cup \mc{B}(k) 
\cup \mc{C}(k)\}$. The converse inclusion is also true.

\begin{theorem}\label{fixed}
 Let $k$ be a fixed integer with $|k|\geq 2$. Then
 \[ \Fix(\mc{T}_k,\C^2) = \{ (x, \bar{x}): x\in 
\mc{A}_k\cup\mc{B}_k\cup\mc{C}_k \}. \]
\end{theorem}
\begin{proof}
It is enough to show that the union $\mc{A}_k\cup\mc{B}_k\cup\mc{C}_k$ has $k^2$ 
elements. In order to do this we will apply the inclusion and exclusion 
principle. 

We start with counting the elements in $\mc{A}_k$. The set $\mc{A}_k$ have 
elements of the form $\alpha(d/(k-1) , e/(k-1))$. It is enough to consider $0 
\leq d,e \leq k-2$ because of the periodicity. There are $(k-1)^2$ such pairs of 
$(d,e)$. These pairs do not result in distinct elements because there are some 
repetitions. A fixed point will be represented six times among these pairs 
unless $d=e$. An element of the form $\alpha(d/(k-1), d/(k-1))$ will be 
represented three times unless $3d\equiv0\pmod{k-1}$. Moreover three corner 
points $\alpha(0,0),\alpha(1/3,1/3)$ and $\alpha(2/3,2/3)$ are in $\mc{A}_k$ if 
$3|k-1$. If $3 \nmid k-1$, then only $\alpha(0,0)$ is in $\mc{A}_k$ among the 
corner points. Thus
\[ |\mc{A}_k| = \frac{(k-1)^2+3(k-1)+2\gcd(k-1,3)}{6}. \]

The set $\mc{B}_k$ have elements of the form $\alpha(d/(k^2-1) , dk/(k^2-1))$. 
It is enough to consider $0 \leq d \leq k^2-2$ because of the periodicity. A 
fixed point will be represented two times among these values unless $d$ is a 
multiple of $k+1$. In that case the representation will be unique. Therefore
\[ |\mc{B}_k| = \frac{(k^2-1)+(k-1)}{2} \]

The set $\mc{C}_k$ consists of elements of the form $\alpha(d/(k^2+k+1) , 
dk/(k^2+k+1))$ with $0 \leq d \leq k^2+k$. Every element is represented three 
times in this case unless $k-1$ is divisible by $3$. Thus
\[ |\mc{C}_k| = \frac{(k^2+k+1)+2\gcd(k-1,3)}{3} \]

Now we consider the intersections. We start with $\mc{A}_k\cap 
\mc{B}_k$. An element of the form $\alpha(d/(k^2-1) , dk/(k^2-1))$ is in 
$\mc{A}_k$ if and only if $d$ is a multiple of $k+1$. There are $k-1$ such 
integers among $\{0,1,2,\ldots,k^2-2\}$, each of which is represented uniquely. 
As a result $|\mc{A}_k\cap\mc{B}_k|=k-1$. The set $\mc{A}_k\cap\mc{C}_k$ may 
only have elements $\alpha(0,0), \alpha(1/3,1/3)$ or $\alpha(2/3,2/3)$ since 
$\gcd(k-1,k^2+k+1)$ divides $3$. Thus $|\mc{A}_k\cap\mc{C}_k|=\gcd(k-1,3)$. 
Similarly $|\mc{B}_k \cap \mc{C}_k| = \gcd(k-1,3)$, and $|\mc{A}_k \cap 
\mc{B}_k 
\cap \mc{C}_k| = \gcd(k-1,3)$. Now it is trivial to verify that $|\mc{A}_k \cup 
\mc{B}_k\cup\mc{C}_k|=k^2$ by applying the inclusion and exclusion principle.
\end{proof}

\section{Value sets over finite fields}
It is a well known fact that the Dickson polynomial $D_k(x)$ induces a 
permutation on $\F_q$ if and only if $\gcd(k,q^s-1)=$ for $s=1,2$. It is in 
perfect analogy with one variable case that the $n$ variable Chebyshev map is a 
bijection of $\F_q^n$ if and only if $\gcd(k,q^s-1)=1$ for $s=1,2,\ldots,n+1$ 
\cite{lidlwells}. In this section we compute the cardinality of 
$V(\mc{T}_k,\F_q^2)$. As a corollary, we recover the result of Lidl and 
Wells in the case $n=2$.

The coefficients of the Dickson polynomials $D_k(x)$ can be computed using the 
following formula:
\[ D_k(x) = \sum_{i=0}^{\lfloor k/2\rfloor} \frac{k(-1)^i}{k-i}\binom{k-i}{i}
 x^{k-2i}.\]
Let $q$ be a power of a prime $p$. It is easily verified using this formula 
that $D_q(x)\equiv x^q \pmod{p}$. Lidl and Wells provide a similar formula for 
the bivariate Chebyshev polynomials \cite[p.~110]{lidlwells}. We have
\[ g_k(x,y) = \sum_{i=0}^{\lfloor k/2\rfloor} \sum_{j=0}^{\lfloor k/3\rfloor} 
\frac{k(-1)^i}{k-i-2j}\binom{k-i-2j}{i+j}  \binom{i+j}{i}x^{k-2i-3j}y^i \]
where only those terms occur for which $k\geq 2i+3j$. Recall that 
$\mc{T}_k(x,y) 
= (g_k(x,y),g_k(y,x))$. It is clear by this formula that 
\[\mc{T}_q(x,y) \equiv (x^q,y^q) \pmod{p}.\]
This congruence enables us to observe that the elements in $\F_q^2$ can be 
obtained by reducing the elements of $\Fix(\mc{T}_q,\C^2)$ modulo a certain 
prime ideal. Because there are precisely $q^2$ fixed points of $\mc{T}_q$, we 
obtain the following lemma.
\begin{lemma}
Let $\F_q$ be a finite field of characteristic $p$. Consider the number field 
$K=\Q(\Fix(\mc{T}_q,\C^2))$ which is obtained by adjoining the coordinates of 
fixed points of $\mc{T}_q$ to the rational numbers. Let $\mf{P}$ be a prime 
ideal of $K$ lying over $p$. Then there exists a one-to-one correspondence
\[ \Fix(\mc{T}_q,\C^2) \longleftrightarrow \F_q^2 \]
which is given by the reduction modulo $\mf{P}$.
\end{lemma}

After characterizing the elements in $\F_q^2$ in a compatible fashion under the 
action of $\mc{T}_q$, determining the cardinality of $V(\mc{T}_k,\F_q^2)$ 
reduces 
to a combinatoric argument.  This is the idea we have used in order to compute 
the size of the value sets for Latt\`{e}s maps \cite{kucuksakalli}.

\begin{theorem}\label{main}
Let $k$ be a nonzero integer and let $\F_q$ be a finite field of characteristic 
$p$. Set
\[a=\frac{q-1}{\gcd(k,q-1)},\ \ b=\frac{q^2-1}{\gcd(k,q^2-1)}\ \ \textnormal{ 
and }\ \ c=\frac{q^2+q+1}{\gcd(k,q^2+q+1)}.\] 
Then the cardinality of the value set is 
\[|V(\mc{T}_k,\F_q^2)|=\frac{a^2}{6}+\frac{b}{2}+\frac{c}{3}+\eta(k,q)\]
where $\eta(k,q)$ is given by
\[\begin{array}{c|c|c}
\eta(k,q) & 3\nmid k \textnormal{ or } 3\nmid a & 3|k \textnormal{ and } 3|a\\ 
\hline
2 \nmid k \textnormal{ or } 2\nmid b & 0 & 2/3\\ \hline
 2|k \textnormal{ and } 2| b & a/2 & a/2+2/3
 \end{array}\]
In particular if $\gcd(k,6)=1$, then $\eta(k,q)=0$.
\end{theorem}
\begin{proof} We have $\Fix(\mc{T}_q, \C^2) = \{ (x, \bar{x}): x\in 
\mc{A}_q\cup\mc{B}_q\cup\mc{C}_q \}$ and there is a one-to-one correspondence 
between $\Fix(\mc{T}_q,\C^2)$ and $\F_q^2$. There will be three types of 
elements $\mc{T}_k(x,y)$ in the value set $V(\mc{T}_k, \F_q^2)$ depending on $x$ 
being in $\bar{\mc{A}}_q, \bar{\mc{B}}_q$ and $\bar{\mc{C}}_q$. We will refer to 
these elements as Type-I, Type-II and Type-III, respectively. The proof of the 
case $k=1$ is similar to the proof of Theorem~\ref{fixed}. Other cases require a 
more detailed investigation. For each type we give the form of $x$ and the 
number of elements in that type by the following table:
\begin{center}
\begin{tabular}{c|c|c}
Type-I & $\alpha\left(\frac{d}{a}, \frac{e}{a}\right)$ & 
$\frac{a^2+3a+2\gcd(a,3)}{6}$\\ \hline
Type-II & $\alpha\left(\frac{d}{b}, \frac{dq}{b}\right)$ & 
$\frac{b+\gcd(b,q-1)}{2}$\\ \hline
Type-III & $\alpha\left(\frac{d}{c}, \frac{dq}{c}\right)$ & 
$\frac{c+2\gcd(c,3)}{3}$ \\ 
\end{tabular}
\end{center}
The number of elements which fit into different types are given by the following 
table:
\begin{center}
\begin{tabular}{c|c|c|c}
Type-I\&II & Type-I\&III & Type-II\&III & Type-I\&II\&III\\ \hline
$\gcd(a,b)$ & $\gcd(a,c)$ & $\gcd(b,c)$ & $\gcd(a,b,c)$
\end{tabular}
\end{center}
Applying the inclusion and exclusion principle we see that the cardinality 
of the value set $V(\mc{T}_k,\F_q^2)$ can be written as
\begin{align*}
 |V(\mc{T}_k,\F_q^2)|=& \left(\frac{a^2}{6}+\frac{b}{2}+\frac{c}{3}\right) + 
\left(\frac{a}{2}+\frac{\gcd(b,q-1)}{2}-\gcd(a,b)  \right) \\
& +\left( \frac{\gcd(a,3)+2\gcd(c,3)}{3}- \gcd(a,c)-\gcd(b,c)+\gcd(a,b,c) 
\right).
\end{align*}
The second term in the above sum is $0$ unless $2|b$ and $2|k$. To see this 
note that if $2|b$ and $2|k$, then $\frac{\gcd(b,q-1)}{2}=\frac{a}{2}$ and 
$\gcd(a,b)=\frac{a}{2}$. If $2\nmid b$ or $2\nmid k$ then $\gcd(a,b)$ 
becomes $a$. A case by case investigation shows that the third term is $0$ 
unless $3|k$ and $3 |a$. If $3|k$ and $3|a$, then each greatest common divisor 
appearing in the third term is equal to $1$ except $\gcd(a,3)=3$. Thus the third 
term of the sum becomes $2/3$.
\end{proof}

We recover the result of Lidl and Wells for bivariate Chebyshev maps by 
Theorem~\ref{main}. More precisely we have a sufficient and necessary condition 
for bivariate Chebyshev maps for being a permutation of $\F_q^2$.

\begin{corollary}\label{maincor}
 The bivariate Chebyshev map $\mc{T}_k(x,y)$ induces a permutation of $\F_q^2$ 
if and only $\gcd(k,q^s-1)=1$ for $s=1,2,3$.
\end{corollary}

We finish our paper by giving an example to illustrate the invariants  
introduced in Theorem~\ref{main}.

\begin{example}
Let $k=6^i$ with $i=0,1,2,\ldots$ and consider the bivariate Chebyshev map 
$\mc{T}_{6^i}$ on $\F_{73}^2$. We have
$\mc{T}_6(x,y)=(g_6(x,y),g_6(y,x))$ where
\[g_6(x,y)=x^6 - 6yx^4 + 9y^2x^2 + 6x^3  - 2y^3 - 12yx  + 3.\]
The maps $\mc{T}_{6^i}$ are not bijections of $\F_{73}^2$ for $i=1, 2, 3, 
\ldots$ since $6$ is not relatively prime to $73-1,73^2-1$ and $73^3-1$. We find 
the cardinality of the value set by using Theorem~\ref{main} and obtain the 
following table.
\[\begin{array}{c|c|c|c|c|c|c|c} 
k & 6^0 & 6^1 & 6^2 & 6^3 & 6^4 & 6^5 & \ldots\\ \hline
a & 72 & 12 & 2 & 1 & 1 & 1 & \ldots\\ \hline
b & 5328 & 888 & 148 & 74 & 37 & 37 & \ldots\\ \hline
c & 5403 & 1801 & 1801 & 1801 & 1801 & 1801 & \ldots\\ \hline
\eta(k,73) & 0 & 12/2+2/3 & 2/2 & 1/2 & 0 & 0 & \ldots\\ \hline
|V(\mc{T}_k,\F_{73}^2)| & 5329 & 1075 & 676 & 638 & 619 & 619 & \ldots 
 \end{array}\]
Note that the size of the value set of $\mc{T}_{6^i}$ will be $619$ for $i \geq 
4$ since $a,b$ and $c$ are relatively prime to $6$ from that point on.
 \end{example}

{\small
\def\refname{References}
\newcommand{\etalchar}[1]{$^{#1}$}


\begin{thebibliography}{tt}



\bibitem[CGM88]{chou}
W.~S.~Chou, J.~Gomez-Calderon, G.~L.~Mullen, \textit{Value sets of Dickson
polynomials over finite fields.} J. Number Theory 30 (1988), no. 3, 334--344. 

\bibitem[Ko74]{koornwinder}
T.H.~Koornwinder, \textit{Orthogonal polynomials in two variables which are 
eigenfunctions of two algebraically independent partial differential 
operators}, III, IV, Indag. Math. 36 (1974), 357--369, 370--381.

\bibitem[K\"{u}14]{kucuksakalli}
\"{O}. K\"{u}\c{c}\"{u}ksakall\i, Value sets of Latt\`{e}s maps over finite 
fields. J. Number Theory 143 (2014), 262--278. 


\bibitem[LN83]{lidlnied}
R.~Lidl, H.~Niederreiter, \textit{Finite fields, Encyclopedia of Mathematics and
its Applications, Vol. 20.} Cambridge, UK: Cambridge University Press, 1983.

\bibitem[LW72]{lidlwells}
R.~Lidl, C.~Wells, \textit{Chebyshev polynomials in several variables.}
J. Reine Angew. Math. 255 (1972), 104--111. 

\bibitem[MWW13]{mullen}
G.L.~Mullen, D.~Wan, Q.~Wang, \textit{Value sets of polynomial maps over finite 
fields.} Q. J. Math. 64 (2013), no. 4, 1191--1196. 


\bibitem[Si07]{sil-dyn}
 J.~H.~Silverman,  \textit{The arithmetic of dynamical systems.}
Graduate Texts in Mathematics, 241. Springer, New York, 2007.

\bibitem[Uc09]{uchimura}
K.~Uchimura, \textit{Generalized Chebyshev maps of $\C^2$ and their 
perturbations.} Osaka J. Math. 46 (2009), no. 4, 995--1017. 


\end{thebibliography}
\end{document}